\documentclass{amsart}
\usepackage{amssymb,latexsym}
\usepackage{amsthm}
\usepackage{amsmath}
\theoremstyle{plain}
\newtheorem{theorem}{Theorem}

\newtheorem{lemma}{Lemma}
\newtheorem{proposition}{Proposition}
\theoremstyle{definition}

\theoremstyle{remark}
\newtheorem{remark}{Remark}

\numberwithin{equation}{section}

\begin{document}
\title[The stable splitting of $\mathit{bu}\wedge{BSO(2n)}$]
      {The stable splitting of $\mathit{bu}\wedge{BSO(2n)}$}
\author{I-Ming Tsai}
\address{DEPARTMENT OF MATHEMATICS\\
              NATIONAL TSING HUA UNIVERSITY\\
              HSINCHU,  TAIWAN}
\email{imingtsai@mx.nthu.edu.tw}
\date{November 7, 2019}
\begin{abstract}
We give the stable splitting of the spectra $\mathit{bu}\wedge{BSO(2n)}$, which completes the question of finding the stable splitting of the complex connective $K$-theory of classifying spaces of special orthogonal groups.
\end{abstract}

\maketitle

\section{Motivation}\label{s:motivation}
The stable splitting of the spectra $\mathit{bu}\wedge{X}$ started with E. Ossa's computation of $\mathit{bu}\wedge\Sigma^{-2}B\mathbb{Z}/2\wedge{B\mathbb{Z}/2}$(see \cite{Os89}), followed by Bruner and Greenlees giving more results on $\mathit{bu}\wedge{BG}$ in \cite{BuGr91}, $BG$ the classifying space of some finite group $G$. For infinite Lie groups, Tsung Hsuan Wu proved that $\mathit{bu}\wedge{BSO(2n+1)}$ also has a stable splitting \cite{Wu18}. In this paper I show that the stable splitting of $\mathit{bu}\wedge{BSO(2n)}$ is in fact equivalent to a wedge sum of spectra related to the splitting of $\mathit{bu}\wedge{BSO(2n-1)}$.

\section{Introduction}\label{s:introduction}
In this paper our cohomology has $\mathbb{Z}/2$ coefficients, so $\tilde{H}^\ast(X)$ would always stand for $\tilde{H}^\ast(X,\mathbb{Z}/2)$. $A$ will denote the mod 2 Steenrod Algebra, $\mathit{bu}$ is the spectrum of the complex connective $K$-theory, with $H^\ast(\mathit{bu})\cong{A//A(Q_{0},Q_{1})}\cong{A\otimes_{E}{\mathbb{Z}/2}}$, $E=\mathbb{Z}/2\langle{Q_{0},Q_{1}}\rangle$ the exterior algebra over the Milnor primitives $Q_{0}$ and $Q_{1}$. $H\mathbb{Z}/2$ is the $\mathbb{Z}/2$ Eilenberg-MacLane spectrum. If $Y$ is a spectrum, then $\Sigma^{m}Y$ is the suspended spectrum increased by $m$ degrees. The tensor product $\otimes$ will stand for $\otimes_{\mathit{Z}/2}$,  so all the spectra and its homotopy equivalences in this paper are 2-localised. It is a standard fact that $H^\ast(BSO(n))= \mathbb{Z}/2[\widehat{\omega_2},\widehat{\omega_3},\cdots,\widehat{\omega_{n}}]$, $\widehat{\omega_i}\in{H^{i}(BSO(n))}$ is the $i$-th Stiefel-Whitney class. We first give our result on the splitting:

\begin{theorem}\label{T:ppppp}
For every $n\geq2$, we have a stable splitting 
\[
  \mathit{bu}\wedge{BSO(2n)}\simeq[{\bigvee_{\alpha}}\Sigma^{\alpha}H\mathbb{Z}/2]\vee[{\bigvee_{\beta}}\Sigma^{\beta}\mathit{bu}]\vee[\mathit{bu}\wedge{MSO_{2n}}], 
\]
where $MSO_{2n}$ is the Thom space constructed from $BSO(2n)$. $\alpha$ are the degress of the free generators of the free $E$-submodule $M_{2n-1}$ of $\tilde{H}^\ast(BSO(2n-1))$, while $\beta$ are the degrees of trivial generators of the $E$-submodule $D_{2n-1}$, with $\tilde{H}^\ast(BSO(2n-1))\cong{M_{2n-1}}\oplus{D_{2n-1}}$ as $E$-modules.
\end{theorem}

Before showing the construction of the splitting we first give some previous related results as background.

\begin{theorem}[Theorem 1.2 of  \cite{WiYa12}]\label{T:second}
For $\forall{n}\geq1$, there is a stable splitting
\[
   \mathit{bu}\wedge{BO(n)}\simeq[{\bigvee_{\alpha}}\Sigma^{\alpha}H\mathbb{Z}/2]\vee[{\bigvee_{\beta}}\Sigma^{\beta}\mathit{bu}]\vee[{\bigvee_{\gamma}}\Sigma^{\gamma}\mathit{bu}\wedge{\mathbb{R}P^{\infty}}], 
\]
where $\alpha=deg{d_{j}}$, $d_{j}$ are the free $E$-generators of $M$. $\beta$ and their degrees corresponds to trivial generators of $D_{1}^{\ast}$. $\gamma$ and their degrees corresponds to the monomials ${\omega_2}^{2m_{1}}{\omega_4}^{2m_{2}}\cdots{\omega_{2l}}^{2m_{l}},2l<n$.
\end{theorem}

One must find the $E$-module structure of ${H}^\ast(BO(n))$ before it can show the stable splitting. Theorem 2 has its roots in Wilson's earlier paper\cite{Wi84}, when he was trying to determine the complex cobordism of $BO(n)$, which gave ${H}^\ast(BO(n))$ as a sum of $E_{j}$-modules with generators, $E_{j}=\mathbb{Z}/2\langle{Q_{0},Q_{1},\cdots,Q_{j-1}}\rangle$ is the exterior algebra of $Q_{i},0\leq{i}\leq{j}\leq{n}$. After that they went on to build a stable map by analyzing each monomial $E$-generator to see if there is a topological map that realises them. The next 2 theorems are attributed to  \cite{Wu18}:\newline

\begin{theorem}[Theorem A of  \cite{Wu18}]\label{T:ds}
For each $n\geq1$, $\tilde{H}^\ast(BSO(2n+1))$ is isomorphic to $D_{2n+1}\oplus{M_{2n+1}}$ as an $E$-module, where $D_{2n+1}$ is a trivial $E$-module with generators $d_{J}=\widehat{\omega_2}^{2m_{1}}\widehat{\omega_4}^{2m_{2}}\cdots,\widehat{\omega_{2n}}^{2m_{n}}$,  $\sum_{i=1}^{n}m_{i}>0$, $m_{i}\geq0$, and $M_{2n+1}$ is a free $E$-module with a basis of $E$-generators $T_{2n+1}=\langle{t_{i}\vert{j\in\Lambda_{2n+1}}}\rangle$.
\end{theorem}

\begin{theorem}[Theorem B of  \cite{Wu18}]\label{T:dss}
For each $n\geq1$, there is a stable splitting 
\[
   \mathit{bu}\wedge{BSO(2n+1)}\simeq[{\bigvee_{\alpha}}\Sigma^{\alpha}H\mathbb{Z}/2]\vee[{\bigvee_{\beta}}\Sigma^{\beta}\mathit{bu}], 
\]
where $\alpha=deg{t_{j}}$ are the degrees of the generators of $M_{2n+1}$, and $\beta$ are the degrees of the trivial generators of $D_{2n+1}$.
\end{theorem}

Some definitions here. In terms of $E$-modules, we mean a $\mathbb{Z}/2$-module $M$ equipped with the module homomorphism $\varphi:E\otimes{M}\rightarrow{M}$, defined by $E$-actions on the elements of $M$,  that is $\varphi(e\otimes{m})=e(m),e\in{E}$. A $E$-module is free if it has a basis of free $E$-generators. A trivial $E$-module has a generating set of elements such that every $E$-action excluding the identity acting on it results in zero. Now to show the splitting of  $\mathit{bu}\wedge{BSO(2n+1)}$, first one uses the Adams spectral sequence \cite{Adams74} to calculate the $E_{2}^{1,\ast}$ term of $\tilde{\mathit{bu}}_{\ast}(BO(n))$, which is the group $Ext_{E}^{1,\ast}(\tilde{H}^{\ast}(BO(n)),\mathbb{Z}/2)\cong{Ext_{E_{\ast}}^{1,\ast}(\mathbb{Z}/2,\tilde{H}_{\ast}(BO(n)))}$, where $E_{\ast}=\mathbb{Z}/2\langle{\xi_{1},\xi_{2}}\rangle$ is the exterior algebra over the Milnor generators $\xi_{1}$ and $\xi_{2}$ of the dual Steenrod algebra $A_{\ast}=\mathbb{Z}/2[\xi_{1},\xi_{2},\cdots]$. The $E$-module structure of $\tilde{H}^{\ast}(BO(n))$ has already been determined in  \cite{WiYa12}, so one can calculate the Ext groups from the bar and cobar resolutions. In \cite{Wu18} one uses the epimorphism 
\[
  Ext_{E_{\ast}}^{1,\ast}(\mathbb{Z}/2,\tilde{H}_{\ast}(BO(2n)))\rightarrow{Ext_{E_{\ast}}^{1,\ast}(\mathbb{Z}/2,\tilde{H}_{\ast}(BSO(2n+1)))}, 
\]
to track the generators in 
\[
  Ext_{E_{\ast}}^{1,\ast}(\mathbb{Z}/2,\tilde{H}_{\ast}(BSO(2n+1)))\cong{Ext_{E}^{1,\ast}(\tilde{H}^{\ast}(BSO(2n+1)),\mathbb{Z}/2)}. 
\]
With the knowledge of $Ext_{E}^{1,\ast}(\tilde{H}^{\ast}(BSO(2n+1)),\mathbb{Z}/2)$, the $E$-module structure of\newline$\tilde{H}^\ast(BSO(2n+1))$ can be determined, and with it the topological splitting of $\mathit{bu}\wedge{BSO(2n+1)}$. Unfortunately this fails for the even case because the homomorphism of the Ext groups 
\[
  Ext_{E_{\ast}}^{1,\ast}(\mathbb{Z}/2,\tilde{H}_{\ast}(BO(2n-1)))\rightarrow{Ext_{E_{\ast}}^{1,\ast}(\mathbb{Z}/2,\tilde{H}_{\ast}(BSO(2n)))}
\]
 is not surjective because there is no stable Becker-Gottlieb transfer of the kind $BSO(2n)\rightarrow{BO(2n-1)}$ (see \cite{BeGo75}).

Luckily the splitting of $\mathit{bu}\wedge{BSO(2n)}$ is actually related to $\mathit{bu}\wedge{BSO(2n-1)}$ through its $E$-module structure, in simple words $\tilde{H}^\ast(BSO(2n-1))$ is an $E$-submodule of $\tilde{H}^\ast(BSO(2n))$. With this special property  we can draw out the general $E$-module structure of $\tilde{H}^\ast(BSO(2n))$ and find a suitable splitting.

\section{$E$-module structure of $\tilde{H}^\ast(BSO(2n))$}\label{s:module}

Before we prove our results some basic machinery must be mentioned. The Milnor primitives $Q_{i}$ are related to the Steenrod Squares \cite{St62} as $Q_{0}=Sq^{1},Q_{1}=Sq^{3}+Sq^{2}Sq^{1}$, so to compute $Q_{i}(\widehat\omega_{k})$ one must use the Wu formula: 

\begin{proposition}[Wu formula \cite{Wu48}]
For the Stiefel-Whitney classes  $\omega_m$, 
\[
  Sq^{k}(\omega_m)=\sum_{t=0}^{k}{m-k+t-1 \choose t}\omega_{k-t}\omega_{m+t},m\geq{k}.
\]
\end{proposition}

Take note that $Q_{i}$ has the property $Q_{i}(xy)=Q_{i}(x)y+xQ_{i}(y)$, which is helpful in our calculations.  For $H^{\ast}(BSO(n))$, if $2m\leq{n}$, then $Q_{0}(\widehat\omega_{2m})=\widehat\omega_{2m+1}$, and $Q_{1}(\widehat\omega_{2m})=\widehat\omega_3\widehat\omega_{2m}+\widehat\omega_{2m+3}$. While for $2m+1\leq{n}$, $Q_{0}(\widehat\omega_{2m+1})=0$,  $Q_{1}(\widehat\omega_{2m+1})=\widehat\omega_3\widehat\omega_{2m+1}$. We now find the $E$-Module structure of $\tilde{H}^\ast(BSO(2n))$. Since this is a polynomial ring and the Stiefel Whitney classes are algebraically independent, it can also be viewed as $\mathbb{Z}/2[\widehat{\omega_2},\widehat{\omega_3},\cdots,\widehat{\omega_{2n-1}}][\widehat{\omega_{2n}}]$. Then ${H}^\ast(BSO(2n))$ and ${H}^\ast(BSO(2n-1))$ are related as polynomial rings 
\[
   {H}^\ast(BSO(2n))\cong\bigoplus_{k\geq0}{H}^\ast(BSO(2n-1))\widehat{\omega_{2n}}^{k}. 
\]
The generators of the $E$-modules $D_{2n-1}$ and $M_{2n-1}$ is described in \cite{Wu18}, but the monomial generators in  $\tilde{H}^\ast(BSO(2n-1))$ might not behave the same way when applied to the $E$-action in $\tilde{H}^\ast(BSO(2n))$, owing to the fact that it has an extra variable $\widehat{\omega_{2n}}$. Luckily Proposition 1 tells us that all the $E$-actions of $\widehat{\omega_{i}},2\leq{i}\leq{2n-1}$ in $\tilde{H}^\ast(BSO(2n-1))$ remain the same as in $\tilde{H}^\ast(BSO(2n))$, so it's an $E$-submodule of $\tilde{H}^{\ast}(BSO(2n))$. We only need to worry about $\widehat{\omega_{2n}}$. $Q_{0}(\widehat{\omega_{2n}})=0$, $Q_{1}(\widehat{\omega_{2n}})=\widehat{\omega_{3}}\widehat{\omega_{2n}}$, and no $\widehat{\omega_{i}}$ can map to $\widehat{\omega_{2n}}$ through the $E$-actions, so this actually says that every  ${H}^\ast(BSO(2n-1))\widehat{\omega_{2n}}^{k}$ is in fact an $E$-submodule of $\tilde{H}^\ast(BSO(2n))$:

\begin{lemma}
For every  $k\geq1, {H}^\ast(BSO(2n-1))\widehat{\omega_{2n}}^{k}$ is an $E$-submodule of\newline $\tilde{H}^\ast(BSO(2n))$, and $\tilde{H}^\ast(BSO(2n-1))$ is an $E$-submodule of $\tilde{H}^\ast(BSO(2n))$.
\end{lemma}
\begin{proof}
For any $x\in\tilde{H}^\ast(BSO(2n-1))$, $Q_{i}(x)$ does not contain any factor of $\widehat{\omega_{2n}}$ in $\tilde{H}^\ast(BSO(2n))$, plus the $E$-actions of $\widehat{\omega_{i}}$ remains the same when they belong to $\tilde{H}^\ast(BSO(2n-1)),\tilde{H}^\ast(BSO(2n))$ respectively. As for $x\widehat{\omega_{2n}}^{k}\in{H}^\ast(BSO(2n-1))\widehat{\omega_{2n}}^{k}$, $Q_{i}(x\widehat{\omega_{2n}}^{k})=(Q_{i}(x)+ik\widehat{\omega_{3}}x)\widehat{\omega_{2n}}^{k}\in{H}^\ast(BSO(2n-1))\widehat{\omega_{2n}}^{k}$, also closed under $E$-actions.
\end{proof}
So one can see that  $\tilde{H}^\ast(BSO(2n))$ is actually an infinite direct sum of $E$-submodules ${H}^\ast(BSO(2n-1))\widehat{\omega_{2n}}^{k}$ and $\tilde{H}^\ast(BSO(2n-1))$. Naturally one would go on and analyze each submodule to decompose them into a sum of $E$-generators, but that is not needed. The main reason is even if we did, then we might not be able to find a suitable spectra representing them. Now $\tilde{H}^\ast(BSO(2n-1))\cong{D_{2n-1}}\oplus{M_{2n-1}}$, and the quotient $\tilde{H}^\ast(BSO(2n))/\tilde{H}^\ast(BSO(2n-1))\cong\bigoplus_{k\geq1}{H}^\ast(BSO(2n-1))\widehat{\omega_{2n}}^{k}={H}^\ast(BSO(2n))\widehat{\omega_{2n}}$ actually corresponds to the cohomology of the Thom space $MSO_{2n}$ of $BSO(2n)$, which by the Thom isomorphism theorem, $\tilde{H}^\ast(MSO_{2n})\cong{H}^\ast(BSO(2n))\mu_{2n}$, $\mu_{2n}$ the unique Thom class of degree $2n$. One can see this by looking at the sphere bundle map $BSO(2n-1)\rightarrow{BSO(2n)}\rightarrow{MSO_{2n}}$, where $MSO_{2n}$ is the mapping cone \cite{Sw75}. Taking the cohomology we have a short exact sequence
\[
  0\rightarrow\tilde{H}^\ast(MSO_{2n})\rightarrow\tilde{H}^\ast(BSO(2n))\rightarrow\tilde{H}^\ast(BSO(2n-1))\rightarrow0
\]
Which sends $\mu_{2n}$ to $\widehat{\omega_{2n}}$. Combined with the facts above, this can be seen as a split exact sequence. All we need to verify is that the $E$-module structure of $\tilde{H}^\ast(MSO_{2n})$ is the same as $\tilde{H}^\ast(BSO(2n))/\tilde{H}^\ast(BSO(2n-1))$.
\begin{lemma}
As $E$-modules, $\tilde{H}^\ast(BSO(2n))\cong{\tilde{H}^\ast(BSO(2n-1))\oplus{\tilde{H}^\ast(MSO_{2n})}}$.
\end{lemma}
\begin{proof}
We need to show the $E$-actions of elements in $\tilde{H}^\ast(MSO_{2n})$ is the same as the corresponding elements in $\tilde{H}^\ast(BSO(2n))\widehat{\omega_{2n}}$. Remember that $\widehat{\omega_{i}}=\Phi^{-1}(Sq^{i}(\mu_{2n}))$. Where $\Phi:{H}^{m}(BSO(2n))\cong{H}^{m+2n}(MSO_{2n})$ is the Thom isomorphism by $\Phi(x)=p^{\ast}(x)\mu_{2n}$ with $p:\gamma^{n}\rightarrow{BSO(2n)}$, so one can see $Sq^{i}(\mu_{2n})$ as $\widehat{\omega_{i}}\mu_{2n}$. For any $z\mu_{2n}\in\tilde{H}^\ast(MSO_{2n}),z\in{H}^\ast(BSO(2n))$, we have $Q_{i}(z\mu_{2n})=Q_{i}(z)\mu_{2n}+zQ_{i}(\mu_{2n})=(Q_{i}(z)+iz\widehat{\omega_{3}})\mu_{2n}$, which behaves exactly the same as any $x\widehat{\omega_{2n}}\in{H}^\ast(BSO(2n))\widehat{\omega_{2n}}$, plus the fact that $\mu_{2n}$ is freely generated over ${H}^\ast(BSO(2n))$ and cannot be generated from ${H}^\ast(BSO(2n))$ by the $E$-actions, so $\tilde{H}^\ast(MSO_{2n})\cong{H}^\ast(BSO(2n))\widehat{\omega_{2n}}$ as $E$-modules.
\end{proof}

\begin{remark}
From Lemma 1, each $E$-submodule ${H}^\ast(BSO(2n-1))\widehat{\omega_{2n}}^{k}$ can be decomposed as sums of $E$-monomial generators. If $k$ is even, then it consists of free generators $t_{i}\widehat{\omega_{2n}}^{k}$ and trivial generators $d_{J}\widehat{\omega_{2n}}^{k}$, $t_{i}\in{M_{2n-1}},d_{J}\in{D_{2n-1}}$. But for $k$ odd, besides the free $E$-generators it contains, it has a set of $E$-generators of the form $t_{j}\widehat{\omega_{2n}}^{k},d_{J}\widehat{\omega_{2n}}^{k}$ which are subject to the relation $Q_{0}Q_{1}=0$. The free and trivial generators have spectra to identify with, but the relational generators does not correspond to any known spectra. This is why we do not decompose the submodules furthur.
\end{remark}

Hence every $E$-generator of ${H}^\ast(BSO(2n))\widehat{\omega_{2n}}$ can correspond to a generator of $\tilde{H}^\ast(MSO_{2n})$, and vice versa. With the $E$-module structure of $\tilde{H}^\ast(BSO(2n))$ settled, the stable splitting can be constructed.

\section{Splitting of the spectra}\label{s:topological}
We now can find the stable splitting of $\mathit{bu}\wedge{BSO(2n)}$. The $E$-module structure tells us certain information about what the splitting of the spectra looks like. But first we have to use Liulevicius's theorem \cite{Liu68} to get from the original cohomology ring to the $E$-module algebra:
\begin{proposition}[Proposition1.7 of \cite{Liu68}]
$A$ is a Hopf algebra, and $B$ is a Hopf subalgebra of $A$, while $M$ and $N$ are left $A$-module and left $B$-modules respectively. We have the isomorphism 
\[
   {}_D\![M\otimes(A\otimes_{E}N)]\cong{}_L\![A\otimes_{B}{}_D\!(M\otimes{N})]
\]
 as left $A$-modules. 
\end{proposition}
So one can easily see that the Steenrod algebra $A$ is a Hopf algebra, and $E$ is a Hopf subalgebra of $A$, the cohomology rings are the left $A$-modules. Now $M,N$ are left $A$-modules with $A$-actions $\mu_{M},\mu_{N}$, $M\otimes{N}$ is also a left $A$-module with $A$-action defined by the following map:
\[
   A\otimes{M}\otimes{N}\rightarrow{A}\otimes{A}\otimes{M}\otimes{N}\rightarrow{A}\otimes{M}\otimes{A}\otimes{N}\rightarrow{M}\otimes{N},
\]
the first is the diagonal map tensored with the identity maps, the second map is the twist map $a\otimes{m}\rightarrow{m}\otimes{a}$ tensored with identities. The notation described above ${}_D\!(M\otimes{N})$ is $M\otimes{N}$ with this twist action. While ${}_L\!(M\otimes{N})$ means the extended $A$-action over $M$ ($A$ only acts on $M$). Also if $A$ is a left $A$-module and a right $B$-module at the same time, then $A\otimes_{B}N$ is a left $A$-module with the extended action over $A$, $N$ a left $B$-module.
\begin{remark}
According to the proposition, if we set $N=\mathbb{Z}/2$ and $B=E$, the isomorphism becomes $\theta:{}_L\![A\otimes_{E}M]\cong{}_D\![(A\otimes_{E}{\mathbb{Z}/2})\otimes{M})]$, and is given by $\theta(a\otimes{m})=\sum{a'\otimes1\otimes{a''m}}$, with inverse $\theta^{-1}(a\otimes1\otimes{m})=\sum{a'\otimes\chi(a'')m}$, where $\psi(a)=\sum{a'\otimes{a''}}$, $\chi$ is the conjugation map. The reader can check the details in \cite{Liu68}.
\end{remark}
Now ${H}^\ast(\mathit{bu}\wedge{X})\cong{{H}^\ast(\mathit{bu})}\otimes{\tilde{H}^\ast(X)}\cong{(A\otimes_{E}{\mathbb{Z}/2})\otimes{\tilde{H}^\ast(X)}}\overset{ \theta^{-1}}\cong{A\otimes_{E}\tilde{H}^\ast(X)}$ for any space $X$, so we replace it with $BSO(2n)$. Using Lemma 2 and Proposition 2 we get 
\[
   {H}^\ast(\mathit{bu}\wedge{BSO(2n)})\cong\newline{A\otimes_{E}\tilde{H}^\ast(BSO(2n))}\cong
\]
\[
   {A\otimes_{E}(\tilde{H}^\ast(BSO(2n-1))\oplus{\tilde{H}^\ast(MSO_{2n})})}\cong
\]
\[
   (A\otimes_{E}D_{2n-1})\oplus(A\otimes_{E}M_{2n-1})\oplus(A\otimes_{E}\tilde{H}^\ast(MSO_{2n})). 
\]
This is where we get to verify our construction of the maps. To prove $\mathit{bu}\wedge{BSO(2n)}$ is equivalent to the wedge sum of spectra described in the theorem, we construct a mapping $\Psi:\mathit{bu}\wedge{BSO(2n)}\rightarrow[{\bigvee_{\alpha}}\Sigma^{\alpha}H\mathbb{Z}/2]\vee[{\bigvee_{\beta}}\Sigma^{\beta}\mathit{bu}]\vee[\mathit{bu}\wedge{MSO_{2n}}]$ such that it induces a mod 2 cohomology isomorphism, hence a homotopy equivalence. This map should be a wedge sum of maps that represent the monomial generators of the $E$-module structure of $\tilde{H}^\ast(BSO(2n))$. Due to the results of Section 3, we can see that there are 3 types of monomial generators which needs a suitable topological realisation, that is the representing spectra whose cohomology exactly are the $E$-modules generated by them. We show the construction of the maps along with the proof of the homotopy equivalence of $\Psi$.\newline
\textit{Proof} \textit{of} \textit{Theorem} \textit{1}. For the module $D_{2n-1}\subset\tilde{H}^\ast(BSO(2n-1))\subset\tilde{H}^\ast(BSO(2n))$, we take the map $\phi\circ{h_{2n}}=\psi_{\beta}:BSO(2n)\rightarrow{\Sigma^{\beta}\mathit{bu}}$ that takes $1\in\tilde{H}^\ast(\Sigma^{\beta}\mathit{bu})$ to each trivial generator $d_{J}$. This map leads to 
\[
   \Psi_{\beta}:\mathit{bu}\wedge{BSO(2n)}\xrightarrow{1\wedge{h_{2n}}}\mathit{bu}\wedge{BO(2n)}\xrightarrow{1\wedge{\phi}}\mathit{bu}\wedge\Sigma^{\beta}\mathit{bu}\rightarrow\Sigma^{\beta}\mathit{bu}. 
\]
The first map $1\wedge{h_{2n}}$ is the smash product of $h_{2n}$ and the identity map of $\mathit{bu}$, $h_{2n}$ is the 2-fold map, the details is described in \cite{MiSt74}. The second map $1\wedge{\phi}$ is the smash product of the identity of $\mathit{bu}$ with $\phi:BO(2n)\rightarrow\Sigma^{\beta}\mathit{bu}$, which takes 1 to a trivial generator, is described in \cite{WiYa12}. The last map is just the multiplication of the ring spectrum $\mathit{bu}$. To check the mapping of generators in cohomology, $1\in\tilde{H}^\ast(\Sigma^{\beta}\mathit{bu})$ to $1\otimes\Sigma^{\beta}1$ then to $1\otimes{d}$, where $d$ is the $E$-trivial generator in $\tilde{H}^\ast(BO(2n))$, and finally through the 2-fold map it sends to $1\otimes{d_{J}}$, where $d_{J}$ is the trivial generator in $\tilde{H}^\ast(BSO(2n-1))$, so $\beta$ is the degree of $d_{J}$. Now from the $A$-module summand related to $d_{J}$ in ${H}^\ast(\mathit{bu}\wedge{BSO(2n)})$ is just $A\otimes_{E}\Sigma^{\beta}\mathbb{Z}/2\cong\tilde{H}^\ast(\Sigma^{\beta}\mathit{bu})$. This cyclic isomorphism map from generator to generator, hence all the trivial generators in $D_{2n-1}$ are represented by the maps.\newline For the free module $M_{2n-1}$, the free $E$-generators $t_{j}$ corresponds to the map $\psi_{\alpha}:BSO(2n)\rightarrow\Sigma^{\alpha}H\mathbb{Z}/2$, whereas in the splitting it becomes:
\[
   \Psi_{\alpha}:\mathit{bu}\wedge{BSO(2n)}\xrightarrow{1\wedge\psi_{\alpha}}\mathit{bu}\wedge\Sigma^{\alpha}H\mathbb{Z}/2\xrightarrow{\nu}\Sigma^{\alpha}H\mathbb{Z}/2. 
\]
The first map is just the smash product map $1\wedge\psi_{\alpha}$, the second map $\nu:\mathit{bu}\wedge{H\mathbb{Z}/2}\rightarrow{H\mathbb{Z}/2}$ is the module spectrum mapping of $H\mathbb{Z}/2$. Inspecting the cohomology mappings we find that the generator of the free $A$-module goes from $1\otimes1$ to $1\otimes{t_{i}}\in{H}^\ast(\mathit{bu}\wedge{BSO(2n-1)})\subset{H}^\ast(\mathit{bu}\wedge{BSO(2n)})$. So we can see the entire map is just actually $A\rightarrow{A\otimes_{E}E}\cong{A}$ which takes free generator to free generator. The isomorphism of the free part is done. Now for the map $\Omega:\mathit{bu}\wedge{BSO(2n)}\rightarrow\mathit{bu}\wedge{MSO_{2n}}$, it is just the smash product of the cofiber map $C_{Bi}:BSO(2n)\rightarrow{MSO_{2n}}$ with the spectrum $\mathit{bu}$. In the cohomology of this map, the generators are mapped $1\otimes{z\mu_{2n}}\rightarrow1\otimes{z\widehat{\omega_{2n}}}\cong{z\widehat{\omega_{2n}}}$, where $z\mu_{2n}$ is an $E$-generator in $\tilde{H}^\ast(MSO_{2n})$ which maps to a generator $z\widehat{\omega_{2n}}\in{H}^\ast(BSO(2n))\widehat{\omega_{2n}}$. So this is a cohomology map from ${H}^\ast(\mathit{bu})\otimes{H}^\ast(MSO_{2n})$ to ${H}^\ast(\mathit{bu})\otimes{H}^\ast(BSO(2n))$ such that the isomorphic image is  ${H}^\ast(\mathit{bu})$ tensored with the $E$-submodule ${H}^\ast(BSO(2n))\widehat{\omega_{2n}}$. So by taking the wedge sum of all the maps constructed above, we have $\Psi=[\bigvee_{\alpha}\Psi_{\alpha}]\vee[\bigvee_{\beta}\Psi_{\beta}]\vee\Omega$, the resulting map induces a cohomology isomorphism which maps $E$-generators to $E$-generators. This gives us 
\[
   H^{\ast}([{\bigvee_{\alpha}}\Sigma^{\alpha}H\mathbb{Z}/2]\vee[{\bigvee_{\beta}}\Sigma^{\beta}\mathit{bu}]\vee[\mathit{bu}\wedge{MSO_{2n}}])\xrightarrow{H^{\ast}(\Psi)}
\]
\[
   H^{\ast}(\mathit{bu}\wedge{BSO(2n)})\cong(A\otimes_{E}D_{2n-1})\oplus(A\otimes_{E}M_{2n-1})\oplus(A\otimes_{E}\tilde{H}^\ast(MSO_{2n}))
\]
\newline So $\Psi$ is a homotopy equivalence. This completes the proof.




\begin{thebibliography}{90}

\bibitem{Adams74}
 J. F. Adams,
  \textit{Stable homotopy and generalized homology, Chicago Lecture notes in Math},
  University of Chicago Press,
  1974.
\bibitem{Adams78}
 J. F. Adams,
  \textit{Inifinite loop spaces},
  Hermann Weyl Lectures, 
  The Institute for Advanced Study,
  Princeton University Press and University of Tokyo Press,
  Princeton, New Jersey,
  1978,
  96-131.
\bibitem{AdPr76}
 J. F. Adams and S. B. Priddy,
  \textit{Uniqueness of BSO},
  Math. Proc. Camb. Phil. Soc.,
  $\mathbf{80}$,
  1976,
  475-509.
\bibitem{BeGo75}
 J. C. Becker and D.H. Gottlieb,
  \textit{The transfer map and fiber bundles},
  Topology,
  Vol. 14,
  Pergamon Press,
  1975,
  1-12.
\bibitem{BuGr91}
 R. R. Bruner and J. P. C. Greenlees,
  \textit{The connective K-Theory of finite groups},
  1991.
\bibitem{Ka81}
R. M. Kane,
  \textit{Operations in connective K-theory},
  Memoirs of the American Mathematical Society,
  vol 34,
  Number 254,
  Providence, Rhode Island, USA,
  1981.
\bibitem{Liu68}
 A. Liulevicius,
  \textit{The cohomology of Massey-Peterson algebras},
  Math. Zeitschr.,
  $\mathbf{105}$,
  1968,
  226-256.
\bibitem{Os89}
 E. Ossa,
  \textit{Connective K-theory of elementary abelian groups},
  In: Kawakubo K.(eds) Transformation Groups. Springer Lecture Notes in Mathmatics,
  vol $\mathbf{1375}$,
  1989,
  269-274.
\bibitem{Ma83}
 H. R. Margolis,
  \textit{Spectra and the Steenrod Algebra},
  North-Holland Mathematical  Library, 
  Elsevier Science Publishers B.V.,
  vol $\mathbf{29}$,
  1983.
\bibitem{Mi58}
 J. Milnor,
  \textit{The Steenrod Algebra and its dual},
  The Annals of Mathematics,
  2nd Ser.,
  vol $\mathbf{67}$,
  No.1.,
  1958,
  150-171.
\bibitem{MiSt74}
 J. Milnor and J. D. Stasheff,
  \textit{Characteristic Classes},
  Princeton University Press and University of Tokyo Press,
  Princeton, New Jersey,
  1974,
  145.
\bibitem{St62}
N. E. Steenrod,
  \textit{Cohomology Operations},
  Princeton University Press,
  Annals of Mathematics Studies,
  No.50,
  1962.
\bibitem{Sw75}
 R. M. Switzer,
  \textit{Algebraic Topology-Homotopy and Homology},
  Die Grundlehren der mathematischen Wissenschaften in Einzeldarstellungen,
  Band 212,
  Springer-Verlag,
  Berlin Heidelberg New York,
  1975.
\bibitem{Wi84}
 W. Stephen Wilson,
  \textit{The complex cobordism of BO(n)},
  J. London Math. Soc. (2),
  29,
  1984,
  352-366.
\bibitem{Wu18}
 T. H. Wu,
  \textit{Stable splittings of the complex connective K-theory of BSO(2n+1)},
  Mathematical Journal of Okayama University,
 $\mathbf{60}$,
  2018,
  73-89.
\bibitem{Wu48}
 W. T. Wu,
  \textit{Classes caract\'{e}ristiques et i-carr\'{e}s d'une vari\'{e}t\'{e}},
  Comptes Rendus,
 $\mathbf{230}$,
  1950,
  508-511.
\bibitem{WiYa12}
 W. Stephen Wilson and D. Y. Yan,
  \textit{Stable splitting of the complex connective K-theory of BO(n)},
  Topology and its Applications,
  $\mathbf{159}$,
  2012,
  1409-1414.

\end{thebibliography}
\end{document}